\documentclass[11pt,reqno]{amsart}
\usepackage[english]{babel}
\usepackage[T1]{fontenc}

\usepackage{fullpage}
\usepackage{amsmath,amscd,amsthm,amsfonts,latexsym}

\usepackage{mathtools}

\theoremstyle{plain}
\newtheorem{theorem}                 {Theorem}      [section]
\newtheorem{proposition}  [theorem]  {Proposition}

\newtheorem{lemma}        [theorem]  {Lemma}
\newtheorem*{theorem*}                 {Theorem} 

\theoremstyle{definition}

\newtheorem{remark}       [theorem]  {Remark}

\def \imto{\hookrightarrow}

\DeclareMathOperator{\R}{R}

\DeclareMathOperator{\trace}{trace}

\DeclareMathOperator{\Div}{div}

\def \ln{\mbox{${\overline{\nabla}}$}}
\def \hr{\mbox{${\mathbb{H}^2\times \mathbb{R}}$}}
\def \r{\mbox{${\mathbb R}$}}

\def \h{\mbox{${\mathbb H}$}}
\def \E{\mbox{${\mathbb E}$}}
\def \s{\mbox{${\mathbb S}$}}

\def \g{\mbox{$g_{\kappa,\tau}$}}

\def \N{\mbox{$\mathcal{N}$}}
\def \Nkt{\mbox{$\mathcal{N}_{\kappa,\tau}$}}
\def \M{\mbox{$\mathcal{M}$}}

\def \gt{\mbox{${\widetilde{g}}$}}
\DeclareMathOperator{\grad}{grad}
\DeclareMathOperator{\Ricci}{Ricci}

\begin{document}

\title{Biconservative surfaces in BCV-spaces}

 \author{Stefano Montaldo}
 
 \address{Universit\`a degli Studi di Cagliari\\
 Dipartimento di Matematica e Informatica\\
 Via Ospedale 72\\
 09124 Cagliari}
 \email{montaldo@unica.it}
 
\author{Irene I. Onnis}
\address{Departamento de Matem\'{a}tica, C.P. 668\\ ICMC,
USP, 13560-970, S\~{a}o Carlos, SP\\ Brasil}
\email{onnis@icmc.usp.br}

\author{Apoena Passos Passamani}
\address{Departamento de Matem\'{a}tica, UFES, 29075-910, Vit\'{o}ria, ES, Brasil}
	\email{apoenapp@gmail.com.br}

 \thanks{
 The work was supported by Regione Autonoma della Sardegna, Visiting Professor Program. The second author was also supported by grant 2015/00692-5, S\~ao Paulo Research Foundation (Fapesp).  The third author was supported by Capes--Brazil}

\begin{abstract}
	{\it Biconservative hypersurfaces} are hypersurfaces with conservative  {\it stress-energy}  tensor with respect to the bienergy functional, and form a geometrically interesting family which includes that of biharmonic hypersurfaces. In this paper we study biconservative surfaces in the $3$-dimensional Bianchi-Cartan-Vranceanu spaces, obtaining their characterization in the following cases: when they form a constant angle with the Hopf vector field; when they are $\mathrm{SO}(2)$-invariant.
\end{abstract}

\maketitle

\section{Introduction}
A hypersurface $\M^{n-1}$ in an $n$-dimensional Riemannian manifold $\N^n$ is called {\it biconservative} if
\begin{equation*}
2 A(\grad f)+ f \grad f=2 f \Ricci(N)^{\top}\,,
\end{equation*}
where $A$ is the shape operator, $f=\trace A$ is the mean curvature function and $\Ricci(N)^{\top}$ is the tangent  component of the Ricci curvature of $\N$ in the direction of the unit normal $N$ of $\M$ in $\N$.\\

The notion of biconservative hypersurfaces was introduced in \cite{CMOP2}, as we shall detail in the next section, where the authors classify, locally, biconservative surfaces into $3$-dimensional space forms (see also \cite{Fuannali}). In \cite{MOR} there is a detailed qualitative study of $SO(p+1)\times SO(q+1)$-invariant  proper biconservative hypersurfaces in the Euclidean space $\r^n$ ($n=p+q+2$). The spacelike and timelike biconservative hypersurfaces of the three-dimensional Minkowski space were studied in \cite{YuFu}, where  the author  gave the local parametrization of the biconservative surfaces that do not have constant mean curvature. Also, in \cite{FOP}, the authors present the classification of the non minimal biconservative surfaces, with parallel mean vector field, in the product spaces $\s^n\times \r$ and $\h^n\times \r$. In \cite{MOR2} there is a motivating study of the relationship between biconservative surfaces and the holomorphicity of a generalized Hopf function. Moreover, the authors give a complete classification of constant mean curvature biconservative surfaces in $4$-dimensional space forms.
\\

In this paper we restrict our study to biconservative surfaces  in the  Bianchi-Cartan-Vranceanu spaces (BCV-spaces). The latter can be thought as a  local representation of simply connected homogeneous three-dimensional Riemannian manifolds and they can be explicitly described by the following two-parameter family of Riemannian metrics:
\begin{equation*}
g_{\kappa,\tau} =\frac{dx^{2} + dy^{2}}{F^{2}} +  \left(dz +
\tau\, \frac{ydx - xdy}{F}\right)^{2},\quad F(x,y)=1+\dfrac{\kappa}{4}(x^2+y^2),\quad \kappa, \tau\in\r,
\end{equation*}
defined on 
 $$\N=\{(x,y,z)\in\r^3\colon F(x,y)> 0\}.$$
 
An important feature of the BCV-spaces is that they admit a Riemannian submersion over a surface with constant Gaussian curvature $\kappa$, called {\it Hopf fibration}:
$$\psi:\N_{\kappa,\tau}= (\N, g_{\kappa,\tau})\to M^2(\kappa)=(\r^2, h=(dx^2+dy^2)/F^2), \quad \psi(x,y,z)=(x,y).$$ 
The vector field $E_3=\partial/\partial z$, which is tangent to the fibers of the Hopf fibration, is called the {\it Hopf vector field}. 
In \cite{Da}, Daniel considered the angle $\alpha$ between the normal vector field of an  immersed surface in $\N_{\kappa,\tau}$ and the Hopf vector field $E_3$, obtaining the expressions of the Gauss and Codazzi equations in terms of the function $\nu=\cos\alpha$. Moreover, he showed that this angle is a fundamental invariant for a surface in a BCV-space.\\

Since the biconservative surfaces in a $3$-dimensional space form have been classified in \cite{CMOP2}, in this paper we study the case when the BCV-space $\Nkt$ is not a space form, that is when $\kappa\neq 4\tau^2$.\\

Using the techniques developed by Daniel, in the first part of the paper we study
biconservative {\it helix surfaces} (or {\it constant angle surfaces}) in the BCV-space $\Nkt$, that is 
surfaces such that the angle $\alpha\in [0,\pi]$ between its unit normal vector field and the unit Killing vector field $E_3$  is constant at every point of the surface. For this class of surfaces we prove the following theorem.\bigskip

{\bf Theorem~\ref{teoACHopf}}. {\it
	Let $\M$ be a non minimal biconservative surface in a BCV-space $\Nkt$, with $\kappa\neq 4\tau^2$. Then, the following statements are equivalent:
	\begin{itemize}
		\item[(a)] $\M$ is a constant angle surface;
		\item[(b)] $\M$ is a CMC surface;
		\item[(c)] $\M$ is a Hopf tube over a curve with constant geodesic curvature.
	\end{itemize}}\bigskip

Since the rotation about the $z$-axes in $\Nkt$ is an isometry for all values of $\kappa$ and $\tau$,  a natural class of surfaces is given by those which are invariant under the action of $SO(2)$ given by rotation about the $z$-axes. These surfaces are called {\it surfaces of revolution}.  The second part of the paper is devoted to the characterization of biconservative surfaces of revolution in $\Nkt$, with $\tau\neq 0$, obtaining the following theorem.\bigskip

{\bf Theorem~\ref{teo:surf-rev-tau-not-zero}}. {\it
Let $\M$ be a surface of revolution in a BCV-space  $\Nkt$, that is not a space form and with $\tau\neq0$. Assume that $f\neq 0$ at every point on $\M$ and $\alpha\in(0,\pi)$. Then, $\M$ is a  biconservative surface if and only if it is a Hopf circular cylinder.   
}\bigskip

\section{Biharmonic maps, stress-energy tensors and biconservative immersions}\label{sec:stress-energy-tensor}

As described by Hilbert in~\cite{H}, the {\it stress-energy}
tensor associated to a variational problem is a symmetric
$2$-covariant tensor $S$ which is conservative at critical points,
i.e. with $\Div S=0$.

In the context of harmonic maps $\varphi:(\M,g)\to (\N,h)$ between two Riemannian manifolds,
that is critical points of the {\em energy} functional
\begin{equation}\label{energia}
E(\varphi)=\frac{1}{2}\int_{\M}\,|d\varphi|^2\,dv_g \,\, ,
\end{equation}
the stress-energy tensor was studied in detail by
Baird and Eells in~\cite{BE} (see also \cite{Sa} and \cite{BR}). Indeed, the Euler-Lagrange
equation associated to the energy functional \eqref{energia} is equivalent to the vanishing of the tension
field $\tau(\varphi)=\trace\nabla d\varphi$ (see \cite{ES}), and the tensor
$$
S=\frac{1}{2}\vert d\varphi\vert^2 g - \varphi^{\ast}h
$$
satisfies $\Div S=-\langle\tau(\varphi),d\varphi\rangle$. Therefore, $\Div S=0$ when the map is harmonic.

\begin{remark}\label{remark:conservative}
We point out that, in the case of isometric immersions, the condition $\Div S=0$ is always satisfied,
 since $\tau(\varphi)$ is normal.
\end{remark}

A natural generalization of harmonic maps are the so-called {\it biharmonic maps}: these maps are the critical points of the bienergy functional (as suggested by Eells--Lemaire \cite{EL83})
\begin{equation}\label{bienergia}
    E_2(\varphi)=\frac{1}{2}\int_{\M}\,|\tau (\varphi)|^2\,dv_g \,\, .
\end{equation}
In \cite{Y2} G.~Jiang showed that the Euler-Lagrange equation associated to $E_2(\varphi)$ is given by the vanishing  of the bitension field
\begin{equation}\label{bitensionfield}
 \tau_2(\varphi) = - \Delta \tau(\varphi)- \trace R^{\tiny{\N}}(d \varphi, \tau(\varphi)) d \varphi  \,\, ,
\end{equation}
where $\Delta$ is the rough Laplacian on sections of $\varphi^{-1} \, (T\N)$ that, for a local orthonormal frame $\{e_i\}_{i=1}^m$ on $\M$, is defined by
$$
    \Delta=-\sum_{i=1}^m\{\nabla^{\varphi}_{e_i}
    \nabla^{\varphi}_{e_i}-\nabla^{\varphi}_
    {\nabla^{\tiny{\M}}_{e_i}e_i}\}\,\,.
$$
The curvature operator on $(\N,h)$, which also appears in \eqref{bitensionfield}, can be computed by means of
$$
    R^{\tiny{\N}} (X,Y)= \nabla_X \nabla_Y - \nabla_Y \nabla_X -\nabla_{[X,Y]} \,\, .
$$

The study of the stress-energy tensor for the
bienergy was initiated  in \cite{Y3} and afterwards developed in \cite{LMO}. Its expression is \begin{eqnarray}\label{eq:stress-bienergy-tensor}
S_2(X,Y)&=&\frac{1}{2}\vert\tau(\varphi)\vert^2\langle X,Y\rangle+
\langle d\varphi,\nabla\tau(\varphi)\rangle \langle X,Y\rangle \\
\nonumber && -\langle d\varphi(X), \nabla_Y\tau(\varphi)\rangle-\langle
d\varphi(Y), \nabla_X\tau(\varphi)\rangle,
\end{eqnarray}
and it satisfies the condition
\begin{equation}\label{eq:2-stress-condition}
\Div S_2=-\langle\tau_2(\varphi),d\varphi\rangle,
\end{equation}
thus conforming to the principle of a  stress-energy tensor for the
bienergy.\\

If  $\varphi:(\M,g)\imto (\N,h)$ is an isometric immersion, then \eqref{eq:2-stress-condition} becomes
\begin{equation}\label{eq:2-stress-condition-tangent}
(\Div S_2)^{\#}=- \tau_2(\varphi)^{\top}\,,
\end{equation}
where $\#$ denotes the musical isomorphism sharp.\\

We say that an isometric immersion is {\em biconservative} if the corresponding 
stress-energy tensor $S_2$ is conservative, i.e. $\Div S_2=0$.

Thus, from  \eqref{eq:2-stress-condition-tangent}, biconservative isometric immersions correspond to immersions with vanishing tangential part of the corresponding bitension field. \\

The decomposition of the bitension field for  hypersurfaces is given in the following theorem (see, for example, \cite{Chen,Ou}). 
\begin{theorem}\label{DecomposicaoBitensao}
	Let $ \varphi:\M^{n-1}\imto \N^n$ be an isometric immersion with unit normal vector field $N$ and mean curvature vector field $H=(f/(n-1))N$. Then, the normal and tangential components of $\tau_2(\varphi)$ are respectively
	\begin{equation*}
		\Delta f +f |A|^2-f \Ricci(N,N)=0
	\end{equation*}
	and
	\begin{equation*}
		2 A (\grad{f})+ f \grad{f}- 2f \Ricci (N)^{\top}=0,
	\end{equation*}
	where $A$ is the shape operator and $\Ricci (N)^{\top}$ is the tangent component of the Ricci curvature of $\N$ in the direction of  the vector field $N$.
\end{theorem}

By Theorem~\ref{DecomposicaoBitensao} an  isometric immersion $\varphi: \M^{n-1} \imto \N^n$ is  {\it biconservative}  if $\varphi$ satisfies the condition 
\begin{equation}\label{bi-conservative}
		2 A (\grad{f})+ f \grad{f}- 2f \Ricci (N)^{\top}=0.
\end{equation}
The hypersurface $\M^{n-1}$ immersed in this way is called a {\it biconservative hypersurface}.

\section{Bianchi-Cartan-Vranceanu spaces}\label{BCV-espacos}

A Riemannian manifold $(\M,g)$ is said to be homogeneous if for every two points $p$ and $q$ in $\M$, there exists an isometry of $\M$, mapping $p$ into $q$. The classification of simply connected $3$-dimensional homogeneous spaces is well-known and can be summarized as follows. The dimension of the isometry group must be equal $6$, $4$ or $3$. If the isometry group is of dimension $6$, $\M$ is a complete real space form, i.e. the Euclidean space $\E^3$, a sphere $\s^3(k)$, or a hyperbolic space $\h^3(k)$. If the dimension of the isometry group is $4$, $\M$ is isometric to $\mathrm{SU}(2)$, the special unitary group, to $\widetilde{\mathrm{SL}(2, R)}$, the universal covering of the real special linear group, to $\mathrm{Nil}_3$, the Heisenberg group, all with a certain left-invariant metric, or to a Riemannian product $\s^2(k) \times \r$ or $\h^2(k) \times \r$. Finally, if the dimension of the isometry group is $3$, $\M$ is isometric to a general simply connected Lie group with left-invariant metric.

\`E.~Cartan classified all $3$-dimensional spaces with $4$-dimensional isometry group in \cite{C}. In particular, he proved that they are all homogeneous and obtained the following two-parameter family of spaces, which are now known as the Bianchi-Cartan-Vranceanu spaces, or BCV-spaces for short. For  $\kappa, \tau\in\r$, we define $\N_{\kappa,\tau}$
as the following open subset of $\r^3$:
 $$
 \N=\{(x,y,z)\in\r^3\;:\; 1+ \frac{\kappa}{4}(x^2+y^2)> 0\},
 $$
equipped with the metric
\begin{equation}\label{metrica}
g_{\kappa,\tau} =\frac{dx^{2} + dy^{2}}{F^{2}} +  \left(dz +
\tau\, \frac{ydx - xdy}{F}\right)^{2},
\end{equation}
where $F=1+\dfrac{\kappa}{4}(x^2+y^2)$. 

The metrics $g_{\kappa,\tau}$, called {\it Bianchi-Cartan-Vranceanu metrics}, can be found for the first time in the classification of the $3$-dimensional homogeneous Riemannian manifolds  given by L.~Bianchi in~1928~(see \cite{B}) and, latter, appeared as \eqref{metrica} in \cite{C} and \cite{V}, thanks to \'{E}.~Cartan and G.~Vranceanu, respectively.

The family \eqref{metrica} of metrics includes all the $3$-dimensional homogeneous metrics whose isometry group has dimension $4$ or $6$, except for the hyperbolic space, according to the following scheme:
\begin{itemize}
\item if $\kappa=\tau=0$, then $\N_{\kappa,\tau}\cong \E^3$;
\item if $\kappa=4\tau^2\neq 0$, then $\N_{\kappa,\tau}\cong \s^3\left(\frac{\kappa}{4}\right)\setminus\{\infty\}$;
\item if $\kappa >0$ and $\tau=0$, then $\N_{\kappa,\tau}\cong (\s^2(\kappa)\setminus\{\infty\})\times\r$;
\item if $\kappa <0$ and $\tau=0$, then $\N_{\kappa,\tau}\cong \h^2(\kappa)\times\R$;
\item if $\kappa >0$ and $\tau\neq 0$, then $\N_{\kappa,\tau}\cong\mathrm{SU}(2)\setminus\{\infty\}$;
\item if $\kappa <0$ and $\tau\neq 0$, then $\N_{\kappa,\tau}\cong\widetilde{\mathrm{SL}(2, R)}$;
\item if $\kappa =0$ and $\tau\neq 0$, then $\N_{\kappa,\tau}\cong \mathrm{Nil}_3$.
\end{itemize}

With respect to the  globally defined orthonormal frame
\begin{equation}\label{eq-basis}
E_1= F\frac{\partial}{\partial x}-\tau y \frac{\partial}{\partial z},\quad
E_2=F\frac{\partial}{\partial y}+\tau x \frac{\partial}{\partial z},\quad
E_3=\frac{\partial}{\partial z},
\end{equation}
the non zero components of the  Ricci curvature are
\begin{equation}\label{Ricci}
 \Ricci(E_1,E_1)=\Ricci(E_2,E_2)=\kappa-2\tau^2,\qquad \Ricci\big(E_3,E_3)=2\tau^2.
\end{equation}

As we have mentioned in the introduction an important feature of the BCV-spaces is that they admit a Riemannian submersion over a surface with constant Gaussian curvature $\kappa$, called the {\it Hopf fibration}:
$$\psi: \N_{\kappa,\tau} \to M^2(\kappa)=\Bigg(\r^2, h=\frac{dx^2+dy^2}{F^2}\Bigg), \quad \psi(x,y,z)=(x,y).$$ 
The vector field $E_3$, which  is tangent to the fibers of this fibration, is called the {\it Hopf vector field}. \\ 

Let now $\M$ be an oriented, simply connected surface in $\Nkt$.
Denote by $\nabla$ the Levi-Civita connection of $\M$, by $N$ its unit normal vector field, by $A$ the shape operator associate to $N$ and by $K$ the Gaussian curvature of the surface. We put
$$g_{\kappa,\tau} (E_3,N)=\cos\alpha,$$
where $\alpha:\M\to [0,\pi]$ is the angle function between the unit normal vector field and the Hopf vector field.
Then, projecting $E_3$ on the tangent plane of  $\M$ we get
\begin{equation}\label{eq:e3-t-alpha-n}
E_3=T+\cos \alpha\, N,
\end{equation}
where $T$ is the tangent part of $E_3$ and satisfies $g_{\kappa,\tau}(T,T)=\sin^2 \alpha$.

The importance of the angle $\alpha$ was emphasized by B.~Daniel in \cite{Da}, where he showed that the expressions of the Gauss and Codazzi equations can be written in terms of the function $\cos\alpha$, as illustrated in the following proposition.

\begin{proposition}[\cite{Da}]\label{GaussCodazzi} 
The Gauss and Codazzi equations for a surface immersed in a BCV-space $\Nkt$
are, respectively,

\begin{equation}\label{gauss1}
	K= \det A+\tau^2+(\kappa-4\tau^2)\cos^2 \alpha,
\end{equation}
	
\begin{equation}\label{codazzi1}
	\nabla_X AY -\nabla_Y AX-A[X,Y]= (\kappa-4\tau^2)\cos \alpha\, ( g_{\kappa,\tau} (Y,T)X
	- g_{\kappa,\tau} (X,T)Y),
\end{equation}
where $X$, $Y$ are tangent vector fields of $\M$.	

	Moreover, for any vector field $X$ tangent to $\M\subset \Nkt$, it holds 
	\begin{equation}\label{compatibilidade} 
	\nabla_X T=\cos \alpha\,(AX-\tau JX),
	\qquad
	g_{\kappa,\tau} (AX-\tau JX,T)=-X(\cos\alpha),
	\end{equation}
	where $JY:=N\wedge Y$, for every $Y \in T\M$.
\end{proposition}

The set $\{T,JT\}$ defines an orthogonal basis of $T\M$ and, for convenience, we consider its orthonormalization:  
\begin{equation}\label{basee_i}
e_1=\frac{T}{\sin\alpha}, \qquad e_2=\frac{JT}{\sin\alpha}, \quad \alpha\in(0,\pi).
\end{equation}

With respect to the frame $\{e_1,e_2\}$, taking into account \eqref{compatibilidade}, and supposing that  $\alpha \in (0,\pi)$, the matrix of the shape operator is
\begin{equation}\label{A1}
A=\begin{pmatrix}
e_1(\alpha) & e_2(\alpha)-\tau \\
\ \ & \ \  \\
e_2(\alpha)-\tau & \lambda \\
\end{pmatrix}\, ,
\end{equation}
while the Levi-Civita connection becomes
\begin{equation}\label{nabla-ei}
\begin{aligned}
&\nabla_{e_1} e_1=\cot\alpha\, (e_2(\alpha)-2\tau)\, e_2 ,\qquad
\nabla_{e_2} e_1=\lambda \cot\alpha \, e_2,\\
&\nabla_{e_1} e_2=-\cot\alpha\, (e_2(\alpha)-2\tau)\, e_1,
\qquad \nabla_{e_2} e_2=-\lambda \cot\alpha \, e_1.
\end{aligned}
\end{equation}

Using  \eqref{A1} and \eqref{nabla-ei}, by a direct computation, we can rewrite the Gauss and Codazzi equations \eqref{gauss1} and \eqref{codazzi1} as follows.

\begin{proposition}
	Let $\M$ be a surface in a BCV-space $\Nkt$ 
	 such that $\alpha \in (0,\pi)$. Then, the Gauss equation \eqref{gauss1} and the Codazzi equation \eqref{codazzi1}  are equivalent to the following equations
\begin{equation}\label{codazzi2}
\left\{
\begin{aligned}
& e_1(e_2(\alpha))+\lambda\cot\alpha\, e_2(\alpha)+\cot\alpha\, e_1(\alpha)(e_2(\alpha)-2\tau)-e_2(e_1(\alpha))=0, \\ \\
& 
\cot\alpha\, [2(e_2(\alpha))^2-\lambda \,e_1(\alpha)-6\tau\, e_2(\alpha)+4\tau^2+\lambda^2]+e_1(\lambda)\\
&-e_2(e_2(\alpha))-(4\tau^2-\kappa)\cos\alpha\sin\alpha=0.		
		\end{aligned}
		\right.
	\end{equation}
	
\end{proposition}

\section{Biconservative helix surfaces in BCV-spaces}\label{ACCMC}

Let $\M$ be an oriented, simply connected surface in $\Nkt$ and let $\{e_1,e_2,N\}$ be the adapted frame of $\Nkt$ along $\M$, where $e_1$ and $e_2$ are described in \eqref{basee_i}. We first decompose \eqref{bi-conservative}, which ensures that $\M$ is biconservative, with respect to the  frame $\{e_1,e_2,N\}$. 
\begin{proposition}\label{pro:biconservative-general}
	Let $\M$ be a surface in a BCV-space $\Nkt$, such that $\alpha \in (0,\pi)$. Then, $\M$ is biconservative if and only if
	\begin{equation}\label{bi-conservative2}
	\left\{
	\begin{aligned}
	& e_1(\lambda+e_1(\alpha))\,(\lambda+3e_1(\alpha))+2e_2(\lambda+e_1(\alpha))\,(e_2(\alpha)-\tau)\\
	&-2(4\tau^2-\kappa)(\lambda+e_1(\alpha))\cos\alpha\sin\alpha=0, \\ \\
	& 2e_1(\lambda+e_1(\alpha))\,(e_2(\alpha)-\tau)+(3\lambda+e_1(\alpha))\,e_2(\lambda+e_1(\alpha))=0.
	\end{aligned}
	\right.
	\end{equation}
\end{proposition}
\begin{proof}
	From
	$$\Ricci (N)^{\top}=\Ricci(N,e_1)\,e_1+\Ricci(N,e_2)\,e_2,$$
	putting $e_1= \sum a_i E_i$, $e_2= \sum b_i E_i$, $N=\sum c_i E_i$ and using \eqref{Ricci}, we obtain 
	\begin{equation*}
	\begin{aligned}
	\Ricci(N)^{\top}=& [(\kappa-2\tau^2)\,(c_1 a_1+c_2 a_2)+2\tau^2c_3a_3]\, e_1 \\
	&+[(\kappa-2\tau^2)\,(c_1 b_1+c_2 b_2)+2\tau^2c_3b_3]\, e_2\\
	= &(4\tau^2-\kappa) \,c_3(a_3 \, e_1 + b_3\, e_2).
	\end{aligned}
	\end{equation*}
Using \eqref{eq:e3-t-alpha-n} in \eqref{basee_i}, we have that $a_3=\sin \alpha$, $b_3=0$ and $c_3=\cos\alpha$, so
	\begin{equation}\label{RicciT}
	\Ricci(N)^{\top}=(4\tau^2-\kappa)\cos\alpha \sin\alpha\, e_1.
	\end{equation}
	Next, from \eqref{A1}, the mean curvature function is
	\begin{equation*}\label{f}
	f=\lambda+e_1(\alpha),
	\end{equation*}
	so
	\begin{equation}\label{grad}
	\grad f= \big[e_1(\lambda)+e_1(e_1(\alpha))\big]\, e_1+ \big[e_2(\lambda)+e_2(e_1(\alpha))\big]\, e_2.
	\end{equation}
	Considering again \eqref{A1},  we obtain
	\begin{equation}\label{Agradf}
	\begin{aligned}
	A(\grad f)= &\big[\big(e_1(\lambda)+e_1(e_1(\alpha))\big)e_1(\alpha)+\big(e_2(\lambda)+e_2(e_1(\alpha))\big)\big(e_2(\alpha)-\tau\big)\big] e_1\\ +&\big[\big(e_1(\lambda)+e_1(e_1(\alpha))\big)\big(e_2(\alpha)-\tau\big)+\big(e_2(\lambda)+e_2(e_1(\alpha))\big)\lambda\big]\, e_2.
	\end{aligned}
	\end{equation}
	The result follows substituting \eqref{RicciT},  \eqref{grad} and \eqref{Agradf} in \eqref{bi-conservative}.
\end{proof}

	We say that a surface in a BCV-space $\Nkt$ is a {\it helix surface}, or a {\it constant angle surface}, if the angle $\alpha\in [0,\pi]$ between its unit normal vector field and the unit Killing vector field $E_3$ (tangent to the fibers of the Hopf fibration) is constant at every point of the surface.\\

We have the following characterization of biconservative  helix surfaces.
\begin{proposition}\label{prop33}
	Let $\M$ be a constant angle surface of a BCV-space $\Nkt$, with angle $\alpha \in [0,\pi]$. If  $\M$ is biconservative, then it has constant mean curvature.    
\end{proposition}
\begin{proof}
	Firstly, we consider the case $\alpha \in (0,\pi)$, $\alpha\neq \pi/2$. Since $\alpha$ is constant, the  matrix \eqref{A1} of the shape operator,  with respect to the frame $\{e_1,e_2,N\}$, becomes
	\begin{equation*}\label{ACte}
	A=\begin{pmatrix}
	0 & -\tau \\
	
	-\tau & \lambda \\
	\end{pmatrix}\, ,
	\end{equation*}
	consequently $$f=\lambda.$$
	The Codazzi equations \eqref{codazzi2} reduce to the only equation
	\begin{equation}\label{conditionlambda}
	e_1(\lambda) +\lambda^2 \cot\alpha+\kappa \cos\alpha\,\sin\alpha+
	4\tau^2\cot\alpha \,\cos^2\alpha=0,
	\end{equation}
	and the biconservative condition \eqref{bi-conservative2} to 
	\begin{equation}\label{bi-conservativeCte}
	\left\{
	\begin{aligned}
	& \lambda\, e_1(\lambda)-2\tau \,e_2(\lambda)-2 \lambda\, (4\tau^2-\kappa)\,\cos\alpha\sin\alpha=0, \\
	& 3\lambda \,e_2(\lambda)=2\tau \,e_1(\lambda).
	\end{aligned}
	\right.
	\end{equation}
	Now, replacing the second equation of \eqref{bi-conservativeCte} in the first and taking into account \eqref{conditionlambda}, we obtain a polynomial equation in $\lambda$ with constant coefficients:
	\begin{equation*}
	6\cot\alpha \, \lambda^4+ [3\sin2\alpha \,(8\tau^2-\kappa)+8\tau^2\cot\alpha\,
	(3\cos^2\alpha-1)]\,\lambda^2 
	-8\tau^2\cos\alpha\,(\kappa\sin\alpha+
	4\tau^2\cot\alpha\cos\alpha)=0.
	\end{equation*}
	It follows that $\lambda$ must be constant, and so $f$. \\
	In the case $\alpha=\pi/2$, \eqref{bi-conservative2} becomes
	\begin{equation*}\label{bi-conservative90}
	\left\{
	\begin{aligned}
	& \lambda \,e_1(\lambda)=2\tau\, e_2(\lambda), \\
	& 3\lambda \,e_2(\lambda)=2\tau \,e_1(\lambda),
	\end{aligned}
	\right.
	\end{equation*}
	which implies that $\lambda$ is constant.\\
	Finally, in the case $\alpha=0,\pi$, we have that  $E_1$ and $E_2$ must be tangents to the surface. Thereby, the distribution determined by $\{E_1,E_2\}$ is integrable and, from the Frobenius Theorem, it must be involutive. Therefore, $\tau=0$ and the surface is one of the following: $\mathbb{H}^2 \subset \hr$, $\mathbb{S}^2 \subset \mathbb{S}^2\times \mathbb{R}$ or $\mathbb{R}^2 \subset \mathbb{R}^3$, and all of them are minimal.
\end{proof}
Since from Proposition~\ref{prop33} biconservative constant angle surfaces are CMC, in the next proposition we describe the CMC biconservative surfaces in $\Nkt$.
\begin{proposition}\label{pro:biconservatice-cmc}
	Let $\M$ be a  non minimal biconservative surface in a BCV-space~$\Nkt$, with $\kappa\neq 4\tau^2$ . If $\M$ is a CMC surface, then it must be a Hopf tube over a curve with constant geodesic curvature. 
\end{proposition}
\begin{proof}
	If $\M$ is a CMC surface, the biconservative condition \eqref{bi-conservative} becomes
	\begin{equation*}\label{bi-conservativeCMC}
	f\Ricci(N)^T=0.
	\end{equation*}
Therefore, since $f\neq 0$, taking into account \eqref{RicciT}, we obtain the condition
\begin{equation}\label{eq:riccit0}
(4\tau^2-\kappa)\cos\alpha \sin\alpha=0.
\end{equation}
Thus, from \eqref{eq:riccit0},  we have one of the following possibilities:
	\begin{itemize}
		\item $4\tau^2-\kappa=0$, i.e. the BCV-space is a space form;
		\item $\sin\alpha=0$, and  the surface is minimal (see the proof of Proposition~\ref{prop33});
		\item $\cos\alpha=0$.
	\end{itemize}
By the hypothesis only  the third case can occur and we analyze it by using ideas given in \cite{Ou,OuZ}. As  $\cos\alpha=0$, it results that  $e_1=E_3$, which means that the surface is tangent to the Hopf vector field (that is, it is a Hopf tube). It turns out that it can be parametrized  by $\phi=\phi(u,v)$ such that the $u$-curves are the integral curves of $e_1$ and the $v$-curves are orthogonal to the $u$-curves, that is they are horizontal curves with respect to the Hopf submersion
	$$\psi: \Nkt\to M^2(\kappa):=\Bigg(\r^2, h=\frac{dx^2+dy^2}{F^2}\Bigg), \quad \psi(x,y,z)=(x,y).$$
Let 
$$\begin{aligned}\beta:&\,I\longrightarrow \N_{\kappa,\tau}\\&s\longmapsto \beta(s)\end{aligned},$$
be a $v$-curve parametrized by arc-length and 
let $\alpha(s)=\psi(\beta(s))$
be the projected curve on $M^2(\kappa)$. 
To end the proof we have to prove that  the geodesic curvature of $\alpha$ is constant. We consider the Frenet frame along $\alpha$ into $M^2(\kappa)$ given by $\{{\bf t}=\alpha', {\bf n}\}$, and the Frenet formulas:
	$$\begin{aligned}&\widehat{\nabla}_{\bf t} {\bf t}=\kappa_g\, {\bf n},\\ &\widehat{\nabla}_{\bf t}{\bf n}=-\kappa_g \, {\bf t},\end{aligned}$$
	where $\widehat{\nabla}$ denotes the connection of $M^2(\kappa)$ and $\kappa_g$ the geodesic curvature of $\alpha$ in $M^2(\kappa)$. Since $\psi$ is a Riemannian submersion, 
	 $$\alpha'=d\psi (e_2), \qquad {\bf n}=d\psi(N),$$ 
	 and $e_2$ and $N$ are horizontal vector fields, it follows that
	 $$\kappa_g=h(\widehat{\nabla}_{\bf t} {\bf t}, {\bf n})=\g(\ln_{e_2}e_2,N)=f.$$
	 As $f$ is constant, we conclude that $\kappa_g$ is constant.
\end{proof}

We can summarize the results of this section in the following theorem.

\begin{theorem}\label{teoACHopf}
	Let $\M$ be a non minimal biconservative surface in a BCV-space $\Nkt$, with $\kappa\neq 4\tau^2$. Then,  the following  statements are equivalent:
	\begin{itemize}
		\item[(a)] $\M$ is a constant angle surface;
		\item[(b)] $\M$ is a CMC surface;
		\item[(c)] $\M$ is a Hopf tube over a curve with constant geodesic curvature.
	\end{itemize}	
\end{theorem}
\begin{remark}
We point out that in the case $\Nkt$ is a space form,
\eqref{bi-conservative} reduces to
$$
2 A (\grad{f})+ f \grad{f}=0
$$
so that all CMC surfaces are biconservative.
\end{remark}

\section{Biconservative surfaces of revolution in BCV-spaces}\label{biconserverot}
The vector field $X= y\, \partial/\partial x- x\,  \partial/\partial y$ is a Killing vector field of 
the BCV-space $\Nkt$  for all values of $\kappa$ and $\tau$. Thus, we can consider
the surfaces in $\Nkt$ which are invariant under the action of the one-parameter group of isometries $G_X$, of $\Nkt$, generated by $X$. 
For convenience, we shall introduce cylindrical coordinates  
\begin{equation*}
\left\{
\begin{aligned}
& x=r \cos\theta, \\
& y=r \sin\theta,\\
& z=z,
\end{aligned}
\right.
\end{equation*}
with $r\geq 0$ and $\theta\in(0,2\pi)$. 
In these coordinates the metric \eqref{metrica} becomes 
\begin{equation*}\label{metrica3}
g_{\kappa,\tau} =\frac{dr^{2}}{F^{2}} +r^2\,\Big(\frac{1+\tau^2\, r^2}{F^2}\Big)\,d\theta^2 +dz^2 -2\frac{\tau \,r^2}{F} \,d\theta dz,
\end{equation*}
where $F=1+\dfrac{\kappa}{4}\,r^2$.
Moreover, the  Killing vector field takes the form
$$X=\frac{\partial}{\partial \theta}.$$
The orbit space of the action of $G_X$ can be identified with
\begin{equation}\label{OrbitSpace}
\mathcal{B}:=\Nkt/G_X=\{(r,z)\in \r^2 : r \geq 0\}
\end{equation}
and
the orbital distance metric of $\mathcal{B}$ (see, for example,  \cite{HH1})  is given by 
	\begin{equation*}
	\gt= \frac{dr^2}{F^2}+\frac{dz^2}{1+\tau^2\,r^2}.
	\end{equation*}
Now, consider a surface of revolution $\M$ that, locally, with respect to the cylindrical coordinates, can be parametrized by
\begin{equation}\label{eq-:surf-revolution}
	X(\theta,s)=(r(s),\theta,z(s)),\quad s\in(a,b)\subset\r\,, \theta\in(0,2\pi)
\end{equation}
and suppose that the profile curve $\gamma(s)=(r(s),z(s))$ is parametrized by arc-length in $(\mathcal{B},\gt)$, so that
\begin{equation}\label{ppca}
	\frac{r'^2}{F^2}+\frac{z'^2}{1+\tau^2 \,r^2}=1.
\end{equation}
From \begin{equation*}\label{XpXt_E}
	\left\{
	\begin{aligned}
		&X_{\theta}=\frac{\partial X}{\partial \theta}=-\frac{r\, \sin \theta}{F}E_1+\frac{r \,\cos \theta}{F}E_2-\frac{r^2 \,\tau}{F}E_3,\\
		& X_s=\frac{\partial X}{\partial s}= r'\,\Big(\frac{\cos \theta}{F}E_1+\frac{\sin \theta}{F}E_2 \Big) + z' \,E_3,
	\end{aligned}
	\right.
\end{equation*}
it results that the unit normal vector field of the surface is given by
\begin{equation*}
N=\frac{\big(-F z' \cos\theta - {r r' \tau \sin \theta}\big)}{F\sqrt{1+\tau^2\,r^2}}\,E_1+\frac{\big(-F z'\,\sin \theta+{r r' \tau \cos\theta}\big)}{F\sqrt{1+\tau^2\,r^2}}\,E_2+\frac{r'}{F\sqrt{1+\tau^2\,r^2}}\,E_3.
\end{equation*}
Consequently
\begin{equation}\label{cospsiInv}
	\cos\alpha= g_{k,\tau}( N, E_3 )=\frac{r'}{F\sqrt{1+\tau^2\,r^2}}.
\end{equation}
In order to use the techniques described in Section~\ref{BCV-espacos}, we compute the vector fields $\{T,JT\}$ with respect to the basis $\{X_{\theta},X_s\}$.
As $T$ is the tangent part of $E_3$, it follows that
\begin{equation*}
	\begin{aligned}
		&g_{k,\tau}(T,X_{\theta})=g_{k,\tau}(E_3,X_{\theta})=-\frac{r^2 \,\tau}{F},\\
		&g_{k,\tau}(T,X_s)=g_{k,\tau}(E_3,X_s)=z'.
	\end{aligned}
\end{equation*}
Then, writing 
\begin{equation}\label{Tab}
	T=a \, X_{\theta}+ b\, X_s,
\end{equation}
we obtain the system
\begin{equation*}
	\left\{
	\begin{aligned}
		&g_{k,\tau}(T,X_{\theta})=a \, g_{k,\tau}(X_{\theta},X_{\theta})+b \, g_{k,\tau}(X_{\theta},X_s),\\
		& g_{k,\tau}(T,X_s)= a \, g_{k,\tau}(X_{\theta},X_s)+ b \, g_{k,\tau}(X_s,X_s),
	\end{aligned}
	\right.
\end{equation*}
whose solution, taking \eqref{ppca} into account, is
\begin{equation}\label{abInv}
	a=-\frac{r'^2\, \tau}{F(1+r^2\,\tau^2)}, \qquad
	b=\frac{z'}{(1+r^2\,\tau^2)}.
\end{equation}
Also, in the basis $\{E_1,E_2,E_3\}$, the expression of $T$ is given by
\begin{equation*}
T=\frac{r'(\tau r  r' \sin\theta +F z' \cos\theta )}{F^2(1+r^2\tau^2)}E_1+\frac{r'(-\tau r r' \cos\theta +F z' \sin\theta )}{F^2(1+r^2\tau^2)}E_2
+\frac{\tau^2 r^2 r'^2+F^2z'^2}{F^2(1+r^2\,z^2)}E_3.
\end{equation*}
Therefore, from \eqref{ppca}, it results that
\begin{equation*}
	JT=N\wedge T=\frac{\tau r r' \cos\theta -F z' \sin\theta }{F\sqrt{1+\tau^2r^2}}E_1+\frac{\tau r r' \sin\theta +F z' \cos\theta }{F\sqrt{1+\tau^2r^2}}E_2.
\end{equation*}
With respect to the basis $\{X_{\theta},X_s\}$, we have that
\begin{equation}\label{JTcd}
	JT=c \, X_{\theta}+ d \, X_s,
\end{equation}
where $c$ and $d$ satisfy the system
\begin{equation*}
	\left\{
	\begin{aligned}
		&g_{k,\tau}(JT,X_{\theta})=c\, g_{k,\tau}(X_{\theta},X_{\theta})+d\, g_{k,\tau}(X_s,X_{\theta}),\\
		& g_{k,\tau}(JT,X_s)= c \, g_{k,\tau}(X_{\theta},X_s)+ d\, g_{k,\tau}(X_s,X_s),
	\end{aligned}
	\right.
\end{equation*}
that is
\begin{equation}\label{cdInv}
	c=\frac{F z'}{r\sqrt{1+\tau^2 r^2}}, \qquad d= \frac{\tau\, r}{\sqrt{1+\tau^2 r^2}}.
\end{equation}
The  mean curvature function $f$ of the surface can be computed using standard techniques of equivariant geometry and it is given (see, for example, \cite{CPR}) by 
\begin{equation}\label{fRed}
		\begin{aligned}
			f(s)&=\Big(\frac{1}{r}-\frac{\kappa}{4}\,r \Big)\sin\sigma+\sigma',\\
		\end{aligned}
	\end{equation}
	where $\sigma(s)$ is the angle that  $\gamma$ makes with the ${\partial}/{\partial r}$ direction.

\begin{remark}\label{obs1}
	For later use, we point out that, using \eqref{ppca},  we obtain the following expressions:
\begin{equation}\label{senS}
\cos\sigma=\frac{r'}{F}\,,\quad \sin\sigma=\frac{z'}{\sqrt{1+\tau^2\,r^2}}.
\end{equation}
\end{remark}

We are now in the right position to study when a surface of revolution, locally parametrized by \eqref{eq-:surf-revolution}, is biconservative.
The first step is to write the conditions of Proposition~\ref{pro:biconservative-general} in this context.

\begin{lemma}
Let $\M$ be a surface of revolution in a BCV-space $\Nkt$, whose mean curvature function is $f$. Assume that $\alpha\in(0,\pi)$. Then, $\M$ is biconservative if and only if the following system is satisfied: 
\begin{equation}\label{bicon1}
\left\{
\begin{aligned}
& f'\,\Big[b\,f-2\tau\,d-2\,(\cos\alpha)'\Big]-2f\,(4\tau^2-\kappa)\cos\alpha\,\sin^2\alpha=0,\\
& f'\,(3d\,f-2\tau\,b)=0.
\end{aligned}
\right.
\end{equation}
\end{lemma}
\begin{proof}
First, taking into account \eqref{Tab} and \eqref{JTcd},  \eqref{bi-conservative2}
can be written as
\begin{equation}\label{bi-conservativeInv}
		\left\{
		\begin{aligned}
			& f'\Big[-2b\, T(\cos\alpha)-2d\Big(JT(\cos\alpha)
			+\tau \sin^2\alpha\Big)+b f\sin^2\alpha\Big]
			-2f(4\tau^2-\kappa)\cos\alpha\sin^4\alpha=0,\\
			&\\
			& f'\Big[-2b\,\Big(JT(\cos\alpha)
			+\tau \sin^2\alpha\Big)+(2d\,\lambda+d\, f)\sin^2\alpha\Big]=0.
		\end{aligned}
		\right.
	\end{equation}
Next, using the expressions
\begin{equation*}\label{TcospsiInv}
	\begin{aligned}
		T(\cos\alpha)=b\, \Big(\frac{r'}{F\sqrt{1+\tau^2r^2}}\Big)',
	\end{aligned}
\end{equation*}
\begin{equation*}\label{JTcospsiInv}
	\begin{aligned}
		JT(\cos\alpha)
		=d\, \Big(\frac{r'}{F\sqrt{1+\tau^2r^2}}\Big)',
	\end{aligned}
\end{equation*}
and that $b^2+d^2=\sin^2\alpha$, system \eqref{bi-conservativeInv} becomes \eqref{bicon1}.
\end{proof}

\begin{theorem}\label{teo:surf-rev-tau-not-zero}
Let $\M$ be a surface of revolution in a BCV-space  $\Nkt$, that is not a space form and with $\tau\neq0$. Assume that $f\neq 0$ at every point on $\M$ and $\alpha\in(0,\pi)$. Then, $\M$ is a  biconservative surface if and only if it is a Hopf circular cylinder.  
\end{theorem}
\begin{proof}
The second equation of system~\eqref{bicon1} occurs if and only if either $f$ is constant or 
\begin{equation}\label{segunda}
	3d\,f-2\tau\,b=0.
\end{equation}
If $f$ is constant, from Theorem~\ref{teoACHopf}, $\M$ is a Hopf tube over a curve with  constant geodesic curvature and, since $\M$ is a surface of revolution, we conclude that it is a  Hopf circular cylinder.
  
Next, assume that $f'\neq 0$ everywhere. Replacing the expression of $b$ and $d$, given in \eqref{abInv} and \eqref{cdInv} respectively, in \eqref{segunda},  we obtain the condition
\begin{equation}\label{biconserve1Inv}
	\tau \,(3f \,r \sqrt{1 + \tau^2 \,r^2} - 2 z')=0.
\end{equation}
Then, since $\tau\neq 0$, 
\begin{equation}\label{fBiconserv}
		f=\frac{2 z'}{3 r \,\sqrt{1 + \tau^2 \,r^2}};
\end{equation}
Now, using \eqref{senS} we can  rewrite \eqref{fBiconserv} as
		\begin{equation}\label{fsen}
			f(s)= \frac{2 \sin\sigma}{3 r}.
		\end{equation}
As $f\neq 0$, we have that $\sin\sigma \neq 0$. Then, comparing \eqref{fsen} with \eqref{fRed} we get
		\begin{equation}\label{Sigma'}
			{\sigma'}= {\sin\sigma}\Big(\frac{\kappa\, r}{4}-\frac{1}{3 r} \Big).
		\end{equation}
From \eqref{cospsiInv}, \eqref{senS} and \eqref{abInv}, we can write
\begin{equation}\label{eq:cos-r'-b}
\cos\alpha = \frac{\cos\sigma}{\sqrt{1 + \tau^2 \,r^2}}\,,\quad r'= F \cos\sigma\,,\quad b=\frac{\cos\sigma}{\sqrt{1 + \tau^2 \,r^2}}\,,\quad
\end{equation}
Next, taking the derivative of $f$ and $\cos\alpha$ with respect to $s$, and replacing the value of $\sigma'$ and $r'$ given in \eqref{Sigma'} and \eqref{eq:cos-r'-b}, we obtain
\begin{equation}\label{eq:cosa1-f1}
f' = \frac{\sin2\sigma}{3\,r}\Big(\frac{\kappa\, r}{4}-\frac{4}{3 r} \Big)\,,\quad (\cos\alpha)'=-\frac{\sin^2\sigma}{\sqrt{1 + \tau^2 \,r^2}}\Big(\frac{\kappa\, r}{4}-\frac{1}{3 r} \Big)-\frac{\tau^2\, r\,(4+\kappa\, r^2)\,\cos^2\sigma}{4(1 + \tau^2 \,r^2)^{3/2}} .
\end{equation}
Finally, replacing \eqref{fsen}, \eqref{eq:cos-r'-b} and  \eqref{eq:cosa1-f1} in  the first equation of \eqref{bicon1}, we obtain the condition
$$
(\kappa-4 \tau ^2)\,f\,(\cos2\sigma-1-2 \tau ^2 r^2)\, \cos\sigma=0.
$$
Therefore, as $\kappa\neq 4\tau^2$, we conclude that $\cos \sigma=0$ which implies, using \eqref{eq:cos-r'-b}, that $r={\rm constant}$. Therefore, from \eqref{fsen}, $f$ must be constant which contradicts the hypothesis that $f'\neq 0$ everywhere.
\end{proof}


\begin{thebibliography}{12}

\bibitem{BE} Baird, P., Eells, J.  {\em A conservation law for harmonic maps}. Geometry Symposium Utrecht 1980, 1--25, Lecture Notes in Math. 894, Springer 1981.

\bibitem{BR} Baird, P., Ratto, A. {\em  Conservation laws, equivariant harmonic maps and harmonic morphisms}.  Proc. London Math. Soc. {\bf 64 }(1992), 197--224.

\bibitem{B} Bianchi, L.  Gruppi continui e finiti. Ed. Zanichelli, Bologna
(1928).

\bibitem{CPR} Caddeo, R., Piu, P.,  Ratto, A. {\em SO(2)-invariant minimal and constant mean curvature surfaces in 3-dimensional homogeneous spaces}. Manuscripta Math. {\bf 87}, (1995) 1--12.

\bibitem{CMOP2} Caddeo, R., Montaldo, S., Oniciuc, C., Piu, P.  {\em Surfaces in three-dimensional space forms with divergence-free stress-bienergy tensor}.  Ann. Mat. Pura Appl. {\bf 193} (2014), 529--550.

\bibitem{C} Cartan, \'{E}.  Le{\c{c}}ons sur la g\'{e}om\'{e}trie des espaces de
Riemann. Gauthier Villars, Paris (1946).

\bibitem{Chen} Chen,  B.-Y. {\em Total mean curvature and submanifolds of finite type}, second edition. Serier in Pure Mathematics, World Scentic, Vol. 27 (2015).

\bibitem{Da} Daniel, B. {\em Isometric immersions into $3$-dimensional homogeneous manifolds}. Comment. Math. Helv. {\bf 82} (2007), 87--131.

\bibitem{ES}  Eells, J.,  Sampson, J. H. {\em Harmonic mappings of Riemannian manifolds. } Amer. J. Math. {\bf 86} (1964), 109--160.

\bibitem{EL83} Eells, J.,  Lemaire, L. {\em  Selected topics in harmonic maps.} CBMS Regional Conference Series in Mathematics, 50. American Mathematical Society, Providence, RI, 1983.

\bibitem{FOP}  Fetcu, D., Oniciuc, C., Pinheiro, A.L. {\em CMC biconservative surfaces in $\s^n\times \r$ and $\h^n\times \r$}. J. Math. Anal. Appl. {\bf 425} (2015), 588--609.

\bibitem{YuFu} Fu, Y. {\em On bi-conservative surfaces in Minkowski 3-spaces}.  J. Geom. Phys. {\bf 66} (2013), 71--79.

\bibitem{Fuannali} Fu, Y. {\em Explicit classification of biconservative surfaces in Lorentz 3-space forms}.  Ann. Mat. Pura Appl. {\bf 194} (2015), 805--822.

\bibitem{H} Hilbert, D. {\em Die grundlagen der physik}. Math. Ann. {\bf 92} (1924), 1--32.

\bibitem{HH1} Hsiang, W.T., Hsiang, W.Y. {\em On the existence of codimension one minimal spheres in compact symmetric spaces of rank 2},  J. Diff. Geom. {\bf 17} (1982), 583--594.

\bibitem{Y3} Jiang, G.Y. {\em The conservative law for 2-harmonic maps between Riemannian manifolds}. Acta. Math. Sinica {\bf 30} (1987), 220--225.

\bibitem{Y2} Jiang, G.Y. {\em 2-harmonic maps and their first and second variation formulas}. Chinese Ann. Math. Ser. A 7,  {\bf 7} (1986), 130--144.

\bibitem{LMO} Loubeau, E., Montaldo, S., Oniciuc, C. {\em The stress-energy tensor for biharmonic maps}.  Math. Z. {\bf 259} (2008), 503--524.

\bibitem{MOR} Montaldo, S., Oniciuc,  C., Ratto, A.  {\em Proper biconservative immersions into the Euclidean space}.  Ann. Mat. Pur. Appl.  {\bf 195} (2016), 403--422.

\bibitem{MOR2}  Montaldo, S., Oniciuc, C.,  Ratto, A. {\em  Biconservative surfaces}. J. Geom. Anal.  {\bf 26} (2016), 313--329.

\bibitem{Ou} Ou,Y. -L.  {\em Biharmonic hypersurfaces in Riemannian manifolds}. Pacific J. Math. {\bf 248} (2010), 217--232.

\bibitem{OuZ} Ou,Y. -L.,  Wang Z.-P. {\em Constant mean curvature and totally umbilical biharmonic surfaces in $3$-dimensional geometries}. J. Geom. Phys. {\bf 61} (2011), 1845--1853.

\bibitem{Sa} Sanini, A. {\em Applicazioni tra variet\`{a} riemanniane con energia critica rispetto a deformazioni di metriche}.  Rend. Mat. {\bf 3} (1983), 53--63.

\bibitem{V} Vranceanu, G. Le{\c{c}}ons de g\'{e}om\'{e}trie differentielle. Ed.
Acad. Rep. Pop. Roum., vol. I, Bucarest (1957).




\end{thebibliography}
\end{document}